\documentclass[12pt, a4paper, english]{article}		
\usepackage[T1]{fontenc}
\usepackage{babel}		
\usepackage{fancyhdr}		
\usepackage[utf8]{inputenc}
\usepackage{amsmath}  
\usepackage{amsthm}		
\usepackage{amsfonts}	
\usepackage{mathtools}
\usepackage{amssymb} 
\usepackage{enumerate} 
\usepackage{graphicx} 
\usepackage{indentfirst}
\usepackage{url}
\newtheorem{teorema}{Theorem}
\newtheorem{lema}[teorema]{Lemma}
\newtheorem{proposicio}[teorema]{Proposition}
\newtheorem{corolari}[teorema]{Corollary}
\newtheorem{athm}{Theorem}

\theoremstyle{definition}
\newtheorem{definicio}[teorema]{Definition}
\newtheorem{nota}[teorema]{Remark}
\newtheorem{exemple}[teorema]{Example}

\newtheorem{question}[teorema]{Question}

\DeclareMathOperator{\C}{C}

\DeclareMathOperator{\GL}{GL}

\DeclareMathOperator{\Hol}{Hol}
\DeclareMathOperator{\Aut}{Aut}
\DeclareMathOperator{\Int}{Int}
\DeclareMathOperator{\Out}{Out}
\DeclareMathOperator{\Ker}{Ker}

\DeclareMathOperator{\Fix}{Fix}

\DeclareMathOperator{\Soc}{Soc}

\DeclareMathOperator{\Z}{Z}
\DeclareMathOperator{\Ret}{Ret}

\newcommand{\bigcdot}{\boldsymbol{\cdot}}

\title{Soluble skew left braces and soluble solutions of the Yang-Baxter equation}
\author{A. Ballester-Bolinches \thanks{Department of Mathematics, Guangdong University of Education, Guangzhou, Guangdong, 510303, People's Republic of China} \thanks{Departament de Matem\`atiques, Universitat de Val\`encia, Dr.\ Moliner, 50, 46100 Burjassot, Val\`encia, Spain, \texttt{Adolfo.Ballester@uv.es}, \texttt{Ramon.Esteban@uv.es}, \texttt{Vicent.Perez-Calabuig@uv.es}}\  \and R. Esteban-Romero\addtocounter{footnote}{-1}\footnotemark \and P. Jim\'enez-Seral\addtocounter{footnote}{1}\thanks{Departamento de Matem\'aticas, Universidad de Zaragoza, Pedro Cerbuna, 12, 50009 Zaragoza, Spain, \texttt{paz@unizar.es}}\and V. P\'erez-Calabuig\addtocounter{footnote}{-3}\footnotemark}

\date{}
\begin{document}
\maketitle

\begin{abstract}
The study of non-degenerate set-theoretic solutions of the Yang-Baxter equation calls for a deep understanding of the algebraic structure of a skew left brace. In this paper, the skew brace theoretical property of solubility is introduced and studied. It leads naturally to the notion of solubility of solutions of the Yang-Baxter equation. It turns out that soluble non-degenerate set-theoretic solutions are characterised by soluble skew left braces. The rich ideal structure of soluble skew left braces is also shown. A worked example showing the relevance of the brace theoretical property of solubility is also presented. 
\end{abstract}



\section{Introduction and statements of results}
The study of set-theoretical solutions of the Yang-Baxter equation (YBE) provide a common framework for a multidisciplinary approach from different  areas including knot theory, braid theory, or operator theory among others. 

A (finite) \emph{set-theoretical solution} of the YBE is a pair $(X,r)$, where $X$ is a (finite) set and $r\colon X \times X \rightarrow X \times X$ is a bijective map satisfying the equality $r_{12}r_{23}r_{12} = r_{23}r_{12}r_{23}$, where $r_{12} = r\times \operatorname{id}_X$ and $r_{23} = \operatorname{id}_X \times\, r$. If $r(x,y) = (y,x)$ for every $x,y \in X$, then $(X,r)$ is said to be \emph{trivial}. 

The main problem in this context is to find and classify all set-theoretical solutions with prescribed properties. The algebraic structure of  skew left braces plays a fundamental role.

A \emph{skew left brace} $B$ is a set endowed with two group structures, $(B,+)$ and $(B,\cdot)$, satisfying the following sort of distributivity property
\begin{equation}
\label{eq:distbraces} a \cdot (b+c) = a\cdot b - a + a\cdot c\quad \forall\, a,b,c\in B
\end{equation}
If $(B,+)$ is abelian, then $B$ is just a Rump's \emph{left brace} (see~\cite{GuarnieriVendramin17} and~\cite{Rump07}). Both operations in a skew left brace $B$ can be related by the so-called star product: $a\ast b = -a + a\cdot b -b$, for all $a,b\in B$. Indeed, if both group operations coincide, equivalently $a\ast b = 0$, for all $a,b\in B$, then $B$ is said to be a \emph{trivial left brace}.  If $X$ and $Y$ are subsets of $B$, then $X \ast Y$ is the subgroup of $(B,+)$ generated by the elements of the form $x \ast y$, for all $x \in X$ and $y \in Y$.

Non-degenerate set-theoretic solutions of the YBE (solutions, for short), i.e. solutions for which both components of $r$ are bijective, lead naturally to a skew left brace structure over the group 
\[G(X,r) = \langle x \in X\,|\, xy = uv,\, \text{if $r(x,y) = (u,v)$}\rangle.\]
$G(X,r)$ is said to be the \emph{structure skew left brace of $(X,r)$}. Furthermore, we can define a solution $(B,r_B)$ associated to every skew left brace $B$ (see~\cite{GuarnieriVendramin17}).

It is abundantly clear that a deep understanding of skew left brace structure properties happens to be essential to comprehend and describe solutions of the Yang-Baxter equation.

One of the most challenging problems in the study of any algebraic structure is the classification of its simple objects. A non-zero skew left brace is said to be \emph{simple} if it does not have a non-trivial quotient. Simple skew left braces have been intensively studied with several results available, although we are still a long way off to complete their classification (see \cite{BachillerCedoJespersOkninski18, BachillerCedoJespersOkninski19, CedoJespersOkninski20}). A natural dual problem is to ask about skew left braces without any proper sub skew braces.  Our first main result gives the answer for finite braces and turns out to be crucial for the study of brace theoretical properties.

\begin{athm}
\label{teo:brides-sensesub}
Let $B$ be a finite skew left brace without proper sub skew braces. Then, $B$ is trivial and isomorphic to a group of prime order.
\end{athm}
  
As the algebraic study of skew left braces involves the interaction of two group structures, notions as nilpotency and solubility instinctively emerge and determine interesting properties of solutions. In particular, nilpotency of skew left braces is introduced to deal with multipermutational solutions, or  solutions that can be retracted into the trivial solution over a singleton after finitely many identification steps (see \cite{CedoSmoktunowiczVendramin19, Gateva-IvanovaCameron12, JespersVanAntwerpenVendramin23}, for example). Another interesting property of solutions of the YBE closely related to multipermutability is decomposability, i.e. solutions that can  be decomposed in a disjoint union of two proper solutions (see~\cite{EtingofSchedlerSoloviev99}). The importance of such properties relies on the search of classes of solutions which can be obtained from solutions of smaller cardinality. This opens the door to bring in  simplicity of solutions (see \cite{CedoOkninski21, Vendramin16}): a solution $(X,r)$ is said to be \emph{simple} if it does not admit any non-trivial epimorphism of solutions $(X,r) \rightarrow (Y,s)$. It  follows that simple solutions are indecomposable (whenever $|X|\neq 2$), and for finite involutive solutions ($r^2 = \text{id}$), every finite  indecomposable involutive solution can be obtained as a dynamical extension of a finite simple solution (see~\cite{CastelliCatinoPinto19} and~\cite{Vendramin16}). Therefore, classifying simple solutions appears to be essential to the major problem of classifying all solutions.

It is well-known that finite minimal simple groups, or non-abelian groups  whose all proper subgroups are soluble, played a key role in the classification of all finite simple groups, which in turn are essential to build finite groups using the extension theory. Soluble groups, which are in the antipodes of the simple groups since they possess a rich normal structure, have a lot to say in this context.

The introduction of a good definition of solubility of skew left braces should provide a good framework for studying minimal simplicity of solutions and should give some hints to classify simple skew left braces. As in the group theory case, it is absolutely necessary that soluble skew left braces have a rich ideal structure. It is also necessary that soluble skew left braces are to skew left braces what soluble groups are to groups.   

Soluble  left braces were defined previously in \cite{BachillerCedoJespersOkninski19}: a left brace $B$ is defined to be soluble if the series $\{B_i\}_{i\in \mathbb{N}}$ trivialises at some $n\in \mathbb{N}$. Here, $B_1 := B$ and $B_{n+1} := B_n \ast B_n$ for every $n \geq 1$. Later on, in~\cite{KonovalovSmoktunowiczVendramin21-erratum, KonovalovSmoktunowiczVendramin21} this definition was extended for general skew left braces.
According to this generalisation, every trivial skew left brace $B$ would be soluble, as $B \ast B = 0$; or equivalently, every group regarded as a trivial brace would be soluble.

A possible solution to overcome such pathology could be obtained by means of abelian subideal series: we say that a skew left brace $B$ is \emph{weakly soluble} if the series $\{B_n\}_{n\in \mathbb{N}}$ reaches $0$ at some $n_0\in \mathbb{N}$, and $B_n/B_{n+1}$ is trivial with an abelian group structure for every $n\in \mathbb{N}$. Unfortunately, this notion does not seem good to get our purposes: there are weakly soluble skew left braces with exactly three ideals and the unique minimal ideal is not an elementary abelian brace (see Section~\ref{sec:exemple}).

Our definition of solubility in  Section~\ref{sec:sol} is inspired on concepts as central nilpotency and commutator ideal which have been recently introduced following ideas of universal algebra  (see  \cite{BonattoJedlicka23, JespersVanAntwerpenVendramin23} and  \cite{BournFacchiniPompili23}, respectively). It aims to provide a representative class of skew left braces with a rich ideal structure. 

Our first goal here is to describe the ideal structure of soluble skew left braces. When studying the ideal structure of a skew left brace, chief series introduced in~\cite{BallesterEstebanPerezC24-JH} turn out to play a key role. We prove:

\begin{athm}
\label{teo:soluble-chief^maximal}
Let $B$ be a soluble skew left brace with a chief series. Then, each chief factor of $B$ is abelian and it is either a Frattini chief factor or a complemented chief factor. Moreover, if $B$ is finite, then
\begin{enumerate}
\item every chief factor is isomorphic to an elementary abelian $p$-group for some prime $p$;
\item each maximal sub skew brace of $B$ has prime power index as a subgroup of~$B$.
\end{enumerate}
\end{athm}

This theorem confirms that the rich ideal structure of soluble skew left braces allows to obtain brace-theoretical versions of some important  classical results of soluble groups. Theorem~B does not hold for weakly soluble skew left braces (see Section~\ref{sec:exemple}).

Our next goal is to introduce soluble solutions as an antithesis of simple solutions. Recall that a homomorphism $f \colon (X,r) \rightarrow (Y,s)$ between solutions is a map $f\colon X \rightarrow Y$ satisfying $s(f\times f) = (f\times f)r$. Assume that $f\colon (X,r) \rightarrow (Y,r)$ is a homomorphism of solutions with $(Y,s)$ a trivial solution. Then, $f$ induces an equivalence relation on $X$: $x \Ker\! f\,\, y$ if, and only if, $f(x) = f(y)$, for every $x,y\in X$. If $X_1 = [x]_{\Ker f} \in X/\Ker f$ is the equivalence class of some $x \in X$, then it follows $X_1 \times X_1$ is $r$-invariant.

\begin{definicio}
Let $(X,r)$ be a solution. Assume that there exists a sequence of subsets $X_{t} \subseteq \ldots \subseteq X_1 \subseteq X_0 = X$ with $X_{t} = \{x_{t}\}$ such that, for every $1\leq i \leq t$, there exist a solution $(Y_i,s_i)$ and an epimorphism of solutions  $f_i \colon (X,r) \rightarrow (Y_i,s_i)$ satisfying
\begin{itemize}
\item $X_i \in X/\Ker f_i$, $1\leq i \leq t$;
\item $f_i(X_{i-1}, r|_{X_{i-1}\times X_{i-1}})$ is a trivial subsolution of $(Y_i, s_i)$, $1\leq i \leq t$.
\end{itemize}
Then, $(X,r)$ is said to be \emph{soluble at $x_{t}$}.
\end{definicio} 

%
%
%
%

We show that solubility of skew left braces characterises soluble solutions.

\begin{athm}
\label{teo:soluble_brace<->soluble_sol}
Let $B$ be a skew left brace. Then,  $B$ is soluble if, and only if, its associated solution $(B,r_B)$ is soluble at $0$.
\end{athm}

\begin{athm}
\label{teo:struct_brace_soluble->soluble}
Let $(X,r)$ be a solution such that $G(X,r)$ is a soluble skew left brace. Assume that $G(X,r)$ has an abelian descending series 
\[ G(X,r) = I_0 \supseteq I_1 \supseteq \ldots \supseteq I_n = 0\]
such that $X \cap I_{n-1}$ is not empty. Then $(X,r)$ is soluble at $x$, for every $x \in X \cap I_{n-1}$.
\end{athm}

Paying heed to the important role of central nilpotency in multipermutation solutions (see \cite{JespersVanAntwerpenVendramin23}), a natural question is to consider whether multipermutability of solutions is a stronger property than solubility of solutions. In a forthcoming paper, we will present results confirming that central nilpotency is to brace solubility what group nilpotency is to group solubility. In this context, the following consequences of Theorems~\ref{teo:soluble_brace<->soluble_sol} and~\ref{teo:struct_brace_soluble->soluble} are illustrative results which foresee such relation.

\begin{corolari}
\label{cor:sol_brid_multip->sol_brid_sol}
Let $B$ be a brace such that $(B,r_B)$ is of finite multipermutational level. Then, $(B,r_B)$ is a soluble solution at $0$.
\end{corolari}

\begin{corolari}
\label{cor:sol_mult->sol_sol}
Let $(X,r)$ be a multipermutation solution. If $\Soc(G(X,r)) \cap X$ is not an empty set, then $(X,r)$ is a soluble solution at $x$, for every $x\in \Soc(G(X,r)) \cap X$.
\end{corolari}

The following example shows that condition $\Soc(G(X,r)) \cap X \neq \emptyset$ in Corollary~\ref{cor:sol_mult->sol_sol} is essential, as some multipermutation solutions can be simple.
\begin{exemple}
Consider the set $X = \{1,2,3\}$. We define a Lyubashenko solution, given by $r(x,y) = (\sigma(y), \tau(x))$, where $\sigma = (1,2,3), \tau = (1,3,2) \in \Sigma_3$. Thus, $\lambda_x = \sigma$ and $\rho_x = \tau$, for every $x \in X$. Then, $(X,r)$ has multipermutational level $1$, as $\Ret(X,r)$ has cardinality $1$. 

On the other hand, assume that $f \colon (X,r) \rightarrow (Y,s)$ is an epimorphism with $Y = \{y_1,y_2\}$ and $s(y_i,y_j) = (\lambda'_{y_i}(y_j),\rho'_{y_j}(y_i))$, for every $1\leq i,j\leq 2$. Without loss of generality, suppose that $f(1) = f(2) = y_1$ and $f(3) = y_2$. Then, 
\[ \begin{array}{l}
 y_1 = f(2) = f(\lambda_1(1)) = \lambda'_{y_1}(y_1) \\
 y_1 = f(1) = f(\lambda_1(3)) = \lambda'_{y_1}(y_2)
 \end{array}\]
Therefore, $\lambda'_{y_1}$ is not bijective and so, $(Y,s)$ is not a non-degenerate solution.

Hence, for every epimorphism $f\colon (X,r) \rightarrow (Y,s)$, either $f$ is an isomorphism or $|Y| = 1$; that is, $(X,r)$ is a simple solution.

\end{exemple}

In Section~\ref{sec:exemple}, we present a worked example suggesting that weak solubility is not the best candidate either to play the role of antithesis to simplicity or to have a good relationship to central nilpotency. It also emphasises the importance of having soluble braces as a class of skew left braces with a rich ideal structure. 

Furthermore, it is shown in \cite[Corollary~4.5]{JespersVanAntwerpenVendramin23} that if $B$ is a centrally nilpotent brace, then every maximal subbrace is an ideal. This is a brace analogue of the group theory known result asserting that in nilpotent groups every maximal subgroup is normal. Note that the converse is also true for finite groups, that is, a finite group is nilpotent if, and only if, every maximal subgroup is normal. 

Our example shows that an analogue of this result in the context of skew left braces does not hold: there exist skew left braces $B$ which are not centrally nilpotent but every maximal subbrace of $B$ is an ideal.


\section{Preliminaries}
\label{sec:prelim}
In this section, we fix notation and introduce basic results that are needed in the paper.  

\emph{From now on, the word brace will mean skew left brace. }

It is well known that the identity elements of the additive group $(B,+)$ and multiplicative group $(B,\cdot)$ of a brace $B$ coincide, and it is denoted by $0$. The product of two elements of a brace will be denoted by juxtaposition. 

Given two subsets $X,Y \subseteq B$, we denote by $[X,Y]_+$ and $[X,Y]_{\bigcdot}$ the commutator of $X$ and $Y$ in $(B,+)$ and $(B,\cdot)$, respectively.

A \emph{subbrace} $S$ of $B$ is a subgroup of the additive group which is also a subgroup of the multiplicative group. A \emph{homomorphism} between two braces $A$ and $B$ is a map $f\colon A \rightarrow B$ satisfying that $f(a + b) = f(a) + f(b)$ and $f(ab) = f(a)f(b)$ for all  $a,b\in A$. The kernel of $f$ is defined as the set $\Ker(f) = \{a \in A \, |\,  f(a) = 0\}$. 
If $f$ is bijective, $f$ is called an isomorphism. We shall say that  braces are \emph{isomorphic}, if there is an isomorphism between them.

Recall that the multiplicative group $(B,\cdot)$ acts on the additive group $(B,+)$ via automorphisms: for every $a\in B$, the map $\lambda_a \colon B \rightarrow B$, given by $\lambda_a(b) = -a + ab$, is an automorphism of $(B,+)$ and the map $\lambda\colon (B,\cdot) \rightarrow \Aut(B,+)$ which sends $a \mapsto \lambda_a$ is a group homomorphism (see \cite[Proposition 1.9]{GuarnieriVendramin17}). This group action relates the two operations on a  brace. For every $a,b\in B$, it holds
\begin{equation}
\label{eq:producte-lambda} a b = a+ \lambda_a(b) \quad \text{and}\quad a+b = a\lambda_{a^{-1}}(b).
\end{equation}
Note that $\Ker \lambda = \{a \in B: ab = a+b,\  \forall\, b \in B\}$ is  a subbrace of $B$ by~\eqref{eq:producte-lambda}.

Moreover, it also provides a characterisation of braces by means of regular subgroups of the holomorph of a group. Given a group $G$, the \emph{holomorph} of $G$ is the semidirect product $\Hol(G) = [G]\Aut(G)$ with the operation given by $(g,\varphi)(h,\psi) = (g\varphi(h),\varphi\psi)$, for every $g,h\in G$, $\varphi,\psi\in \Aut(G)$. A \emph{regular subgroup} of the holomorph $H \leq \Hol(G)$ is a subgroup such that for every $g \in G$, there exists a unique $\phi_g\in \Aut(G)$ such that $(g,\phi_g)\in H$. If $g \in G$, the map $^gx = \alpha_g (x) = gxg^{-1}$ is an automorphism of $G$ called the \emph{inner automorphism associated with $g$}. The set $\Int(G)$ of all inner automorphisms of $G$ is a normal subgroup of $\Aut(G)$. In particular, $[G]\Int(G) \leq \Hol(G)$.

\begin{proposicio}[{\cite[Theorem 4.2]{GuarnieriVendramin17}}]
\label{prop:brida-holomorph}
If  $(B,+,\cdot)$ is a brace, then $H_B = \{(b, \lambda_b): b\in B\}$ is a regular subgroup of $\Hol(B,+)$ isomorphic to~$(B,\cdot)$.

Conversely, suppose that for a group $(B,+)$, we have a regular subgroup $H\leq \Hol(B,+)$. Then we can define on $B$ a binary operation $bc := b + \phi_b(c)$, being $(b,\phi_b)\in H$,  such that $(B,+,\cdot)$ becomes a brace and $(B,\cdot)$  is isomorphic to~$H$.
\end{proposicio}

A brace $B$ is said to be \emph{almost trivial} if $ab = b + a$ for every $a,b\in B$. 

\begin{proposicio}
\label{prop:caract-trivials}
Let $(B,+,\cdot)$ be a brace. Then
\begin{enumerate}
\item $B$ is trivial if, and only if, $\Ker \lambda = (B,\cdot)$ for every $b\in B$ or, equivalently, $H_B = \{(b, 1): b\in B\} \leq \Hol(B,+)$. 
\item $B$ is almost trivial if, and only if, $H_B = \{(b, \alpha_{-b}): b\in B\}$. In this case, $\Int(B,+)$ is a homomorphic image of $H_B$.
\end{enumerate}
\end{proposicio}

\begin{proof}
Only the second statement is in doubt. Note that, for all $a,b\in B$, $ab = a + \lambda_a(b) =  b+a$ if, and only if, $\lambda_a(b) = -a+b+a$ or, equivalently, $\lambda_a = \alpha_{-a}$, for every $a\in B$. Moreover, the projection $(b, \alpha_{-b}) \mapsto \alpha_{-b}$ over the second component provides a homomorphism between $H_B$ and $\Int(B,+)$, as 
\[ (b, \alpha_{-b})(a,\alpha_{-a}) = (b -b + a +b, \alpha_{-b}\alpha_{-a}) = (a+b,\alpha_{-(a+b)}).\qedhere\]
\end{proof}

We see in the following how braces can be defined from bijective $1$-cocycles associated with actions of groups. Let $(C,\cdot)$ and $(B,+)$ be groups such that $C$ acts on $B$ by means of a group homomorphism $\lambda \colon C \rightarrow \Aut(B,+)$, $c\mapsto \lambda_c$. A bijective map $\delta \colon C \rightarrow A$ is said to be a \emph{bijective $1$-cocycle} associated with $\lambda$, if $\delta(cd) = \delta(c) + \lambda_c(\delta(d))$, for every $c,d\in C$. In the previous situation, $(B,+)$ admits a brace structure by means of $ab := \delta(\delta^{-1}(a)\delta^{-1}(b))$, for every $a,b\in B$ (see \cite[Proposition 1.11]{GuarnieriVendramin17}), for example).

Following \cite{BallesterEsteban22}, bijective $1$-cocycles can be constructed by means of trifactorised groups. Assume that a group $(C,\cdot)$ acts on a group $(B,+)$ by means of a group homomorphism $\lambda\colon C \rightarrow \Aut(B,+)$. Take $G = [B]C$ the semidirect product associated with $\lambda$. If $D$ is a subgroup of $G$ such that $G = BC = BD = DC$ and $B \cap D = D\cap C = 1$, then $G$ is said to be a \emph{trifactorised group}  and there exists a bijective $1$-cocycle $\delta \colon C \rightarrow B$, given by $D = \{(\delta(c),c): c\in C\}$.


A non-empty subset $I$ of a brace $B$ is a \emph{left  ideal}, if $(I,+)$ is a subgroup of $(B,+)$ and $B \ast I \subseteq I$, or equivalently $\lambda_b(I) \subseteq I$, for every $b\in B$. A left ideal $I$  is an \emph{ideal} if $(I,+)$ is a normal subgroup of $(B,+)$ and $(I,\cdot)$ is a normal subgroup of $(B,\cdot)$. By  \cite[Lemma~1.9]{CedoSmoktunowiczVendramin19}, a left ideal $I$ is an ideal of $B$ if, and only if, $(I,+)$ is a normal subgroup of $(B,+)$ and $I\ast B \subseteq I$.

Ideals of skew left braces can be considered as true analogues of normal subgroups in groups and ideals in rings (see \cite[Lemma 2.3]{GuarnieriVendramin17}, for example). In particular, kernels of homomorphisms of braces are ideals.

\begin{proposicio}
\label{prop:product_id-subbr}
Let $B$ be a brace and let $S$ and $I$ be, respectively, a subbrace and an ideal of $B$. Then $SI = S+I$ is a subbrace of~$B$.
\end{proposicio}

\begin{proof}
Let $x \in S$ and $y\in I$. Then, by~\eqref{eq:producte-lambda}, $xy = x + \lambda_x(y) \in S+I$ and $x + y = x\lambda_x^{-1}(y) \in SI$. Thus $IS = SI = S+I = I+S$ is a subbrace of~$B$.
\end{proof}

\begin{definicio}
Let $B$ be a brace and $S$ be a subbrace of $B$. We say that $S$ is \emph{complemented} in $B$ if there exists a subbrace $T$ such that $B = ST = S+T$ and $S \cap T = 0$. 
\end{definicio}

The following are particular examples of, respectively, of a left ideal and an ideal of a brace such that play a central role in the study of braces:
\begin{align*}
\Fix(B) & = \{a \in B\,|\, \lambda_b(a) = a, \ \text{for every $b\in B$}\}\\
\Soc(B) & = \Ker\lambda \cap \Z(B,+) = \{a\in B\,|\, ab = a+b = b+a, \ \text{for every $b\in B$}\}.
\end{align*}

Every detailed study of an algebraic structure depends heavily on describing the behaviour of respectively minimal and maximal substructures, if such exist, and their quotient structure if apply.
\begin{enumerate}
\item Let $I$ be an ideal of a brace $B$. $I$ is called a \emph{minimal ideal} of $B$ if $I \neq 0$ and $0$ and $I$ are just the ideals of $B$ contained in $I$. On the other hand, $I$ is called a \emph{maximal ideal} of $B$ if $I$ is the only proper ideal of $B$ containing~$I$.
\item Let $S$ be a subbrace of a brace $B$. $S$ is called a \emph{maximal subbrace} of $B$ if $S$ is the only proper subbrace of $B$ containing~$S$.
\end{enumerate}

\section{Braces without proper subbraces}
\label{sec:without_sub}
\emph{In this section every group is finite.}

We work toward a proof of Theorem~\ref{teo:brides-sensesub}. The proof of this result depends heavily on an exhaustive study of automorphisms of simple and characteristically simple groups. 

Recall that a characteristically simple group $G = S^n$ is isomorphic to a direct product of isomorphic copies of a simple group $S$. The next result describes the automorphism group of such groups.

\begin{teorema}
\label{teo:aut-caractsimpl}
Let $G$ be a characteristically simple group with $G = S^n$, $S$ a simple group.
\begin{enumerate}
\item (\cite{Weir55}) If $S = C_p$ is a cyclic group of prime order $p$, then $\Aut(S) \cong \GL(n,p)$. Moreover,  the set of upper triangular matrices with $1$s on the diagonal is a Sylow $p$-subgroup of $\GL(n,p)$.
\item (\cite[Proposition 1.1.20]{BallesterEzquerro06}) If $S$ is not abelian, $\Aut(G) \cong \Aut(S) \wr \Sigma_n$, where $\Sigma_n$ is the symmetric group of degree~$n$.
\end{enumerate}
\end{teorema}

\begin{nota}
\label{nota:Int&Out-nonsimple}
If $S$ is not abelian and $G = S^n$, then $\Z(G) = 1$ and $\Int(G) \cong G$. Moreover, $\Out(G) = \Aut(G)/\Int(G) \cong \Out(S) \wr \Sigma_n$.
\end{nota}

In \cite[Theorem A]{BallesterEstebanJimenez22-maximalsubgroups} we find  a lower bound for the order of the outer automorphism group of a non-abelian simple group.

\begin{teorema}
\label{teo:cotaout}
Let $S$ be a non-abelian simple group. Denote by $\operatorname{l}(S)$ the smallest index of a core-free subgroup of $S$. Then, $|\Out(S)| \leq 3 \log(\operatorname{l}(S))$. 
\end{teorema}

\begin{corolari}
\label{cor:outs<s}
For every simple group $S$,  $|\Out(S)| < |S|$.
\end{corolari}

The following result about fixed-point-free automorphisms turns out to be useful here. 

\begin{teorema}[{\cite[Theorem 1.48]{Go82}}]
\label{teo:fixpoint-autsimple}
A group having an automorphism leaving only the identity element fixed is necessarily soluble.
\end{teorema}

Bearing in mind Proposition~\ref{prop:caract-trivials}, the following lemma turns out to be crucial in the  proof of Theorem~\ref{teo:brides-sensesub}.
\begin{lema}
\label{lema:subgrups-GIntG}
Let $G$ be a non-abelian simple group and consider the semidirect product $X = [G]\Int(G)$. Then, $\bar{G} = G \times 1$ and $C = \{(a, \alpha_{a^{-1}}): a\in G\}$ are exactly the  regular subgroups of $X$ isomorphic with~$G$.
\end{lema}

\begin{proof}
For every $a,b,g\in G$, it holds
\[
(a, \alpha_b)(g,1)(\alpha_{b^{-1}}(a^{-1}), \alpha_{b^{-1}})  = (a\,^bg, \alpha_b)(\alpha_{b^{-1}}(a^{-1}), \alpha_{b^{-1}}) = (^{ab}g, 1)\]
Then, $^{ab}g =  g$ for every $g\in G$ if, and only if, $a = b^{-1}$, as $\Z(G) = 1$. Therefore, $\C_X(\bar{G}) = C = \{(a, \alpha_{a^{-1}}): a \in G\}$, which is a normal subgroup of $X$.

Then $X = \bar{G}C$ as a direct product of $\bar{G}$ and $C$ and both subgroups are isomorphic to $G$. In fact, the map $\tau\colon G \rightarrow C$ given by $\tau(g) = (g^{-1}, \alpha_g)$ is an isomorphism between $G$ and $C$. 

Assume, arguing by contradiction, that there exists $H \leq X$ isomorphic to $G$ such that $H \neq \bar{G}, C$. Note that every element $h \in H$ can be uniquely written  as $(g,1)(c, \alpha_{c^{-1}})$ for some $g, c \in G$. Since $H$ is a simple group, $H \cap \bar{G} = H \cap C = 1$, and then, the projections $\pi_{G}\colon H \rightarrow G$ and $\pi_{C}\colon H \rightarrow C$, given by  $\pi_G((g,1)(c,\alpha_{c^{-1}})) = g$ and $\pi_C((g,1)(c,\alpha_{c^{-1}})) = (c,\alpha_{c^{-1}})$, respectively, are isomorphisms. Thus, we can write
\begin{equation}
\label{eq:2}
 H = \{(g,1)\pi_C\pi_G^{-1}(g): g \in G\}
\end{equation}
and, therefore, $H$ induces the automorphism $\varphi = \tau^{-1} \circ \pi_C \circ \pi_G^{-1}\in \Aut(G)$. Then, \eqref{eq:2} can be written as
\begin{align*}
 H & =  \{(g,1)\tau\varphi(g): g \in G\} =  \{(g,1)(\varphi(g^{-1}), \alpha_{\varphi(g)}): g \in G\}\\
&  = \{ (g\varphi(g^{-1}), \alpha_{\varphi(g)}): g\in G\}
\end{align*}

Since $H$ is a regular subgroup with the same order as $G$, $g\varphi(g^{-1}) \leftrightarrow \alpha_{\varphi(g)}$ provides a one-to-one correspondence between the first and the second component of the elements of $H$. Therefore, the automorphism $\varphi$ has not non-trivial fixed points, as $\varphi(g) = g$ implies that $g\varphi(g^{-1}) = 1 = 1 \varphi(1)$ and so $g = 1$. This contradicts Theorem~\ref{teo:fixpoint-autsimple}. 
\end{proof}

In \cite{Byott04} it is proved that every brace with multiplicative group a simple group must be trivial or almost trivial by means of Hopf-Galois extensions. For the sake of completeness, we give here an alternative proof for the case in which additive and multiplicative groups are isomorphic simple groups.

\begin{lema}
\label{lema:brides2simples}
Let $B$ be a brace such that $(B,+)$ and $(B,\cdot)$ are isomorphic simple groups. The following hold:
\begin{enumerate}
\item if $(B,+)$ is abelian, then $B$ is a trivial brace.
\item if $(B,+)$ is not abelian, then $B$ is either a trivial brace or an almost trivial brace.
\end{enumerate}
\end{lema}

\begin{proof}
By Proposition~\ref{prop:brida-holomorph}, $H = \{(b, \lambda_b): b \in B\} \cong (B, \cdot)$ is a regular subgroup of $\Hol(B,+)$. 

If $(B,+)$ is abelian, then $(B,+)$ cyclic of prime order and, therefore, $\Ker \lambda \neq 0$ as $\Ker \lambda = 0$ would imply $(B,+) \cong (B, \cdot) \cong \Aut(B,+)$. Thus, $\Ker\lambda = B$ and so $B$ is trivial.

Assume that $B$ is not abelian and $B$ is not a trivial brace. Then $\lambda(B)$ is a simple subgroup of $\Aut(B,+)$. Since $\lambda(B) \cap \Int(B,+) \unlhd \lambda(B)$, it follows that either $\lambda(B) \cap \Int(B,+) = 0$ or $\lambda(B) \leq \Int(B,+)$. In the first case, $\Out(B,+)$ has a subgroup of order $|B|$, which is not possible by Corollary~\ref{cor:outs<s}.
Then $\lambda(B)  = \Int(B,+) \cong (B,+)$ and $H$ is a regular subgroup of the semidirect product $[(B,+)]\Int(B,+)$. By Lemma~\ref{lema:subgrups-GIntG}, either $H = \{(b,1): b \in B\}$ or $H = \{(b, \alpha_{-b}): b\in B\}$. Since $B$ is not trivial, the latter holds. Thus $B$ is  almost trivial by Proposition~\ref{prop:caract-trivials}.
\end{proof}

We prove now Theorem~\ref{teo:brides-sensesub}.

\begin{proof}[Proof of Theorem~\ref{teo:brides-sensesub}]

If $B$ is a trivial brace, it follows that the additive group $(B,+)$ has not proper subgroups and, therefore, $B$ is isomorphic to a cyclic group of prime order.

Assume that $B$ is not a trivial brace. We split the proof in several steps towards a contradiction. 

\medskip

\emph{Step 1. $(B,+)$ is a characteristically simple group.}

Assume that $(K,+)$ is a proper characteristic subgroup of $(B,+)$. Then, for every $b\in B$, $\lambda_b(K) \subseteq K$. It is well known that in this case $K$ is a proper subbrace of $B$, which is a contradiction.

Therefore $(B,+)$ is a characteristically simple group and then we can assume that $(B,+) = S^n$ is a finite direct product of $n$ copies of a simple group $S$, for some $n\geq 1$. 

\medskip

\emph{Step 2. $(B, \cdot)$ is isomorphic to a subgroup $H$ of $\Aut(B,+)$. As a consequence, $S$ is not abelian.}

Recall that $\Ker\lambda$ is a subbrace of $B$ and so $\Ker\lambda = 0$. Therefore, $(B, \cdot)$ is isomorphic to a subgroup of $\Aut(B,+)$.

Suppose that $S$ is a cyclic group of  prime order $p$. Then $(B,+) = C_p^n$ is an elementary abelian $p$-group. If $n = 1$, then $|\Aut(B,+)| < p = |(B,\cdot)|$, which is not possible. Thus $n > 1$. By Theorem~\ref{teo:aut-caractsimpl}, we can assume $\Aut(B,+) = \text{GL}(n,p)$. Let $\mathcal{T}$ be the set of upper triangular matrices with $1$s on the diagonal. Then $\mathcal{T}$ is a Sylow $p$-subgroup of $\text{GL}(n,p)$. Since $\lambda(B)$ is a $p$-subgroup, it follows that $\lambda(B) \leq \mathsf{A}\mathcal{T}\mathsf{A}^{-1}$ for some element $\mathsf{A} \in \text{GL}(n,p)$.
Then $L = \{(k,0,\ldots,0): k \in C_p\}$ is a subgroup of $(B,+)$ that is invariant by the action of $\mathcal{T}$, as $\mathsf{T}(L) = L$ for every $\mathsf{T}\in \mathcal{T}$. Consequently $\mathsf{A}(L)$ is a $\lambda$-invariant subgroup of $(B,+)$ and so $L$ is a proper subbrace of $B$, contrary to assumption.  

Hence $S$ is a non-abelian simple group and, in this case, $\Aut(B,+) \cong \Aut(S)\wr \Sigma_n$ and $\Int(B,+) \cong (B,+)$, by Theorem~\ref{teo:aut-caractsimpl}.

Write $H = \lambda(B) \leq \Aut(B,+)$.

\medskip

\emph{Step 3. Either $H \cap \Int(B,+) = 0$ or $H = \Int(B,+)$.}

Note that $f\alpha_xf^{-1} = \alpha_{f(x)}$ for every $\alpha_x\in \Int(B,+)$ and $f\in \Aut(B,+)$. Let $\alpha_x \in \Int(B,+)\cap H$ and $\lambda_b\in H$, for some $x,b\in B$. Since $H \cap \Int(B,+)$ is a normal subgroup of $H$, it follows $\lambda_b\alpha_x\lambda_b^{-1} = \alpha_{\lambda_b(x)} \in H \cap \Int(B,+)$. Thus $H \cap \Int(B,+)$ is isomorphic to a $\lambda$-invariant subgroup of $(B,+)$, which is a subbrace of $B$. Therefore either $H \cap \Int(B,+) = 0$ or $H  = \Int(B,+)$.

\medskip

\emph{Step 4. $H \cap \Int(B,+) \neq 0$.}

Assume that $H \cap \Int(B,+) = 0$. Then, $H$ is isomorphic to a subgroup of $\Out(B,+)$. By Remark~\ref{nota:Int&Out-nonsimple}, we have that $ \Out(B,+) \cong \Out(S) \wr \Sigma_n$. Then $|H| = |B| = |S|^n$ divides $|\Out(S)|^n n!$. 

Observe that $|S|$ does not divide $|\Out(S)|$, as $|\Out(S)| < |S|$ by Corollary~\ref{cor:outs<s}.
Then we can find a prime $p$ such that $p^\alpha \mid |S|$ but $p^\alpha \nmid |\Out(S)|$, for some $\alpha\geq 1$. Suppose that $p^{\beta}$ is the highest power of $p$ dividing $|\Out(S)|$, for some $0 \leq \beta < \alpha$. Then $p^{n(\alpha - \beta)}\mid n!$ and  this is a contradiction because $p^n$ does not divide $n!$. Hence $(B, \cdot) \cong H = \Int(B,+) \cong S^n$.

\medskip

\emph{Step 5. $n = 1$.}

Since $H \cong S^n$, we have that there exists $N \unlhd H$, with $N \cong S$. As in Step~3, it follows that there exists a $\lambda$-invariant subgroup of $(B,+)$ isomorphic to $S$. Therefore $S^n = S$. 

Thus, $(B,\cdot) \cong (B,+) \cong S$ and by Lemma~\ref{lema:brides2simples}, $B$ must be almost trivial. In this case, since $B$ has not proper subbraces, it follows that $(B,+)$ has not proper subgroups. Hence $(B,+)$ is a cyclic group and then $B$ is trivial, which provides the final contradiction.
\end{proof}

\section{Soluble skew left braces}
\label{sec:sol}

In this section, we study the ideal and subbrace structure of braces and introduce a definition of solubility that turns out to be the natural framework to undertake such study. 
To this end it is convenient to have a useful definition of a commutator of ideals that enables us to get a tighter grip on the nature of solubility by constructing canonical abelian series for such braces. In this context, the definition of commutator introduced by Bourn, Facchini and Pompili in \cite{BournFacchiniPompili23} turns out to be the correct one.

\begin{definicio}
Let $I,J$ be ideals of a brace $B$. The \emph{commutator} $[I,J]_B$ is defined to be as the smallest ideal of $B$ containing $[I,J]_+$, $[I,J]_{\bigcdot}$, and the set 
\[ \{ij - (i+j):\, i \in I, \, j \in J\}.\]
It follows that  $[I,J]_B = [J,I]_B$.
\end{definicio}

In \cite{BonattoJedlicka23}, a brace-theoretical analogue of the centre of a group is studied. It was first defined in \cite{CatinoColazzoStefanelli19} in the context of ideal extensions of braces. The \emph{centre of a brace} (also known as \emph{annihilator of a brace}) is the set
\[ \zeta(B) = \Soc(B) \cap \Fix(B) = \{a \in B \, |\, a + b = b + a = ab = ba, \ \forall\, b \in B\},\]
which turns out to be an ideal of~$B$.

The relationship between commutators and centres is enshrined in: 

\begin{proposicio}
\label{prop:central-commut}
Let $I, J$ be ideals of a brace $B$ with $J \subseteq I$. Then, $I/J \subseteq \zeta(B/J)$  if, and only if, $[I,B]_B \subseteq J$. 
\end{proposicio}

\begin{proof}
For every $x\in I$ and $b\in B$, note that 
\[ x+b + J = b+x + J = xb + J = bx + J =  (xb)J = (bx)J\]
if, and only if, $[x,b]_+, [x,b]_{\bigcdot}, xb - (x+b) \in J$. Hence the result follows.
\end{proof}


The commutator concept enables us to construct a canonical abelian series for braces. If $I$ is an ideal of a brace $B$, we define the \emph{commutator} or \emph{derived ideal} of $I$ with respect to $B$ as $\partial_B(I) = [I,I]_B$. If $B = I$, then we say that $\partial_B(B) = \partial (B)$ is the \emph{commutator} or \emph{derived ideal} of $B$. A brace $B$ is said to be \emph{abelian} if $\partial(B) = 0$,  that is, $B$ is a trivial brace with an abelian group structure or, equivalently, $\zeta(B) = B$. 

By repeatedly forming derived ideals, we obtain a descending sequence of ideals
\[ \partial_0(I) = I \supseteq \partial_1(I) = \partial_B(I) \supseteq \ldots \supseteq \partial_n(I) \supseteq \ldots\]
with $\partial_{n}(I) = \partial_B(\partial_{n-1}(I))$ for every $n\in \mathbb{N}$. We call this series the \emph{derived series} of $I$ with respect to $B$. Note that it is an \emph{abelian series}, as each factor $\partial_{n-1}(B)/\partial_{n}(B)$ is an abelian brace for every $n\in \mathbb{N}$ by Proposition~\ref{prop:central-commut}. In particular, the first factor $B/\partial(B)$ is the greatest abelian quotient in $B$.

\begin{definicio}
We say that an ideal $I$ of $B$ is  \emph{soluble with respect to $B$}, if there exists a non negative integer $n$ such that $\partial_n(I) = 0$. In the case $I = B$, we simply say that $B$ is a \emph{soluble brace}.
\end{definicio}

\begin{nota}
Let $B$ be a soluble brace. The shortest length of an abelian series at $B$ is called the \emph{derived length} of $B$. By convention, $B$ has derived length $0$ if, and only if, $B = 0$. It is clear then that abelian braces are soluble braces of derived length $1$. It also follows that soluble braces are closed for subbraces, quotients, direct products and extensions of ideals: if $I$ an ideal of a brace $B$ such that $B/I$ is soluble and $I$ is soluble with respect to $B$, then $B$ is also soluble.
\end{nota}

\begin{proposicio}
\label{prop:abseries-derseries}
Let $B$ be a brace admitting an abelian series, that is, a series of ideals of $B$,
\[ B = I_0 \supseteq I_1 \supseteq \ldots \supseteq I_n = 0,\]
such that $I_{i-1}/I_i$ is abelian for every $1 \leq i \leq n$. Then $\partial_i(B) \subseteq I_i$ for every $0 \leq i \leq n$. In particular, $\partial_n(B) = 0$ so $B$ is soluble. Therefore the derived length of $B$ is the length of the derived series of~$B$.
\end{proposicio}

\begin{proof}
For $i = 0$, $\partial_0(B) = B = I_0$. Now, assume that $\partial_i(B) \subseteq I_i$ for some $0\leq i < n$.  Then, $\partial_{i+1}(B) \subseteq \partial_B(I_i) \subseteq I_{i+1}$ as $I_i/I_{i+1}$ is abelian. Hence, the length of every abelian series of $B$ is less than or equal to the length of the derived series of~$B$.
\end{proof}

It follows at once from this that no abelian series of ideals can be shorter than the derived series of a soluble brace. As a consequence, simple soluble braces and minimal quotients of soluble braces are characterised.
\begin{corolari}
\label{cor:soluble-simple}
Let $B$ be a simple soluble brace. Then, $B$ is an abelian brace isomorphic to a cyclic group of prime order.
\end{corolari}

\begin{corolari}
\label{cor:soluble-maxid}
Let $I$ be a maximal ideal of a soluble brace $B$. Then, $B/I$ is an abelian brace isomorphic to a cyclic group of prime order. 
\end{corolari}

Our focus now is the impact of these notions on the ideal and subbrace structure of a brace.  In particular, we work towards a proof of Theorem~\ref{teo:soluble-chief^maximal}.

Our next result is an application of Theorem~\ref{teo:brides-sensesub}.

\begin{teorema}
\label{teo:max-centvprime}
Let $B$ be a finite brace  and $S$ be a maximal subbrace of $B$. Then, either $\zeta(B) \subseteq S$ or $S$ is an ideal of $B$ such that $B/S$ is an abelian brace isomorphic to a cyclic group of prime order. 
\end{teorema}

\begin{proof}
Assume that $\zeta(S)$ is not contained in $S$. By Proposition~\ref{prop:product_id-subbr}, $S\zeta(B) = S + \zeta(B)$ is a subbrace of $B$, and so $B = S\zeta(B)$ by maximality of $S$. We will show that $S$ is an ideal of $B$. Let $x\in S$ and $b = s+z= sz \in B$, for some $z\in \zeta(B)$ and $s \in S$. Then
\begin{align*}
\lambda_b(x) & =  \lambda_{sz}(x) = \lambda_s(\lambda_z(x)) = \lambda_s(x) \in S;\\
x \ast b & =  x \ast (s+z) = -x + x(s+z) - (s+z) = -x + xs -x + xz -z - s \\
& =  -x + xs - x + x + z - z -s = x \ast s \in S.
\end{align*}
Furthermore, since $B = S + \zeta(B)$, it follows that $(S, +)$ is a normal subgroup of $(B, +)$. Consequently $S$ is an ideal of $B$. 

Therefore, $B/S$ is a brace without subbraces and, by Theorem~\ref{teo:brides-sensesub}, $B/S$ is a trivial brace isomorphic to a cyclic group of prime order. Hence $B/S$ is abelian.
\end{proof}

\begin{definicio}
The \emph{Frattini} subbrace $\Phi(B)$ of a brace $B$ is the intersection of all maximal subbraces of $B$, if such exist, and $B$ otherwise. 
\end{definicio}

Let $S$ be a subset of a brace $B$. The subbrace (resp. ideal) generated by $S$ is the smallest subbrace (resp. ideal) of $B$ containing~$S$. We say that an element $a\in B$ is \emph{non-subbrace (resp. non-ideal) generating} if the subbrace (resp. ideal) generated by $S \cup \{a\}$ coincides with the subbrace (resp. ideal) generated by $S$ for every $S \subseteq B$. It is easily observed that $\Phi(B)$ coincides with the set of non-subbrace generating elements of~$B$.

\begin{nota}
Non-ideal generating elements were introduced in~\cite{JespersKubatVanAntwerpenVendramin21} within the context of brace-theoretical analogues of normal closures and weights in group theory. Then, the \emph{radical} of a brace $B$ was defined as the intersection of all maximal ideals, if such exists, or $B$ otherwise. It turns out that the radical of a brace coincides with the set of non-ideal generating elements.
\end{nota}

\begin{corolari}
Let $B$ be a finite brace  and $S$ be a maximal subbrace of $B$. Then, either $\zeta(B) \subseteq S$ or $\partial(B) \subseteq S$. In particular, $\zeta(B) \cap \partial(B)$ is contained in $\Phi(B)$.
\end{corolari}

However, in general the Frattini subbrace of a brace is not an ideal, not even a left ideal.

\begin{teorema}
There exists a brace $B$ such that $\Phi(B)$ is not a left ideal.
\end{teorema}
\begin{proof}
Let
\begin{align*}
  (K,{+})&=\langle a,b\mid 5a=4b=0, b+a-b=3a\rangle,\\
  (C,{\cdot})&=\langle c,d \mid c^5=d^4=1, dcd^{-1}=c^3\rangle
\end{align*}
be two  Frobenius groups of order~$20$. Observe that $C$ acts on $K$ via
\begin{align*}
  c(a)&=a,&d(a)&=3a,\\
  c(b)&=2a+b,&d(b)&=b.
\end{align*}

Consider $G=[K]C$  the semidirect product of $K$ by $C$ with respect to this action. Let $D=\langle ac, bd\rangle$, then $G=KD=DC$, $D\cap C=D\cap K=\{1\}$ and so we can define a structure of a brace $B$ associated with this triply factorised group. The corresponding bijective $1$-cocycle $\delta\colon C\longrightarrow K$ is described in Table~\ref{tab-der-20}.
\begin{table}[htbp]
  \centering
  \caption{Bijective $1$-cocycle associated with the brace of order~$20$}
  \label{tab-der-20}
  \medskip

  $\begin{array}{cccccccccc}
     c&\delta(c)&c&\delta(c)&c&\delta(c)&c&\delta(c)&c&\delta(c)\\\hline
    1&0&c&a&c^2&2a&c^3&3c&c^4&4c\\
    d&b&cd&3a+b&c^2d&a+b&c^3d&4a+b&c^4d&2a+b\\
    d^2&2b&cd^2&4a+2b&c^2d^2&3a+2b&c^3d^2&2a+2b&c^4d^2&a+2b\\
    d^3&3b&cd^3&2a+3b&c^2d^3&4a+3b&c^3d^3&a+3b&c^4d^3&3a+3b\\\hline
  \end{array}$
\end{table}

We prove by induction that $\delta(c^id^j)=3^jia+jb$. First of all, we prove by induction that $j(2a+b)=(3^j-1)a+jb$: the result is true for $j=0$, and if $j(2a+b)=(3^j-1)a+jb$, then $(j+1)(2a+b)=(2a+b)+(3^j-1)a+jb=2a+3(3^j-1)a+(j+1)b=(3^{j}-1)a+jb$. The result is true for $i=0$, as $\delta(d)=b$ and $d(b)=b$. Suppose that $\delta(c^id^j)=3^jia+jb$. Consequently, $\delta(c^{i+1}d^j)=\delta(c)+c(\delta(c^id^j))=a+c(3^jia+jb)=a+3^jia+c(jb)=a+3^jia+j(2a+b)=a+3^jia+(3^j-1)a+jb=3^j(i+1)a+jb$.

Since the subgroup of index~$2$ of $K$ is characteristic in~$K$, we have that this subgroup corresponds to a maximal subbrace of~$B$. This is the unique maximal subbrace of~$B$ of order divisible by~$5$.

Suppose now that a subbrace of $B$ has order not divisible by~$5$. We have that $\langle b\rangle\le K$, that induces $\langle d\rangle\le C$, leads to a subbrace of $B$ that corresponds to both a maximal subgroup of $K$ and a maximal subgroup of~$C$. Consequently, $\langle b\rangle\le K$ induces a maximal subbrace of~$B$. Now suppose that $c^id$ belongs to the multiplicative group $E$ of a subbrace of $B$. Then $\delta(c^id)=3ia+b$ belongs to the additive group $L$ of that subbrace. Consequently, $c^idc^id=c^ic^{3i}d^2=c^{4i}d^2$ belongs to $E$ and is its unique element of order~$2$. Furthermore $\delta(c^{4i}d^2)=9\cdot 4ia+2b=ia+2b\in L$ is its unique element of order~$2$, but it must coincide with $(3ia+b)+(3ia+b)=3ia+9ia+2b=12ia+2b=2ia+2b$. It follows that $ia=0$, that is, $5\mid i$. Consequently, no element of order~$4$ apart from $d$ and $d^3$ can belong to the multiplicative group of the subbrace. Now suppose that we have a subbrace of order~$2$. The element of order~$2$ of the multiplicative subgroup $E$ must be of the form $c^id^2$, and so $\delta(c^id^2)=9ia+2b=4ia+2b$ is the element of order~$2$ of its additive  subgroup~$L$. But then the corresponding subbrace is not maximal, as $E$ is contained in $\langle c, d^2\rangle$ and $L$ is contained in $\langle a, 2b\rangle$.

It follows that the unique maximal subbraces of $B$ are the ones corresponding to the additive subgroups $\langle a, 2b\rangle$ and $\langle b\rangle$. Thus, $\Phi(B)$ corresponds to the subgroup $\langle 2b\rangle\le K$, associated with $\langle d^2\rangle \le C$. But $\langle 2b\rangle$ is not invariant under the action of $C$, because $c(2b)=2a+b+2a+b=8a+2b=3a+2b\notin \langle 2b\rangle$. Therefore, $\Phi(B)$ is not a left ideal of~$B$.

This brace corresponds to \texttt{SmallSkewbrace(20, 21)} of the \textsf{YangBaxter} library \cite{VendraminKonovalov23-YangBaxter-0.10.3} of \textsf{GAP} \cite{GAP4-13-1}.
\end{proof}

In \cite{BallesterEstebanPerezC24-JH} it is shown that chief factors of braces turn out to play a key role in the ideal structure of braces. Let $I$ and $J$ be ideals of a brace $B$ such that $J \subseteq I$. $I/J$ is said to be a \emph{chief factor} of $B$ if $I/J$ is a minimal ideal of $B/J$. A series of ideals
\[ 0 = I_0 \subseteq I_1 \subseteq \ldots \subseteq I_n = B\]
is said to be a \emph{chief series} if $I_j/I_{j-1}$ is a chief factor of $B$, for every $1 \leq j \leq n$.  A chief factor $I/J$ is a \emph{Frattini chief factor} of $B$ if $I/J \subseteq \Phi(B/J)$ and a \emph{complemented chief factor} of $B$ if it is complemented in~$B/J$.

Observed that if $I$ is a minimal ideal of a soluble brace $B$, then and so $\partial_B(I)$ is an ideal of $B$ properly contained in $I$. Therefore $\partial_B(I) = 0$ and $I$ is abelian. Consequently every chief factor of a soluble brace $B$ is abelian. 

\begin{lema}
\label{lema:max<->factppal}
Let $S$ and $I$ be a subbrace and an ideal of a brace~$B$ respectively. Assume that $I$ is an abelian brace and $B = SI = S+I$. Then 
\begin{enumerate}
\item The subbrace $I\cap S$ is an ideal of~$B$.
\item $S$ is a maximal subbrace if, and only if, $I/I\cap S$ is a chief factor of~$B$. 
\end{enumerate}
\end{lema}

\begin{proof}
Let $b\in B$ and $x \in I \cap S$. Since $B = SI$, we can write $b = sy$ for certain $s\in S$ and $y\in I$. Then
\[ \lambda_b(x) = \lambda_{sy}(x) = \lambda_s(\lambda_y(x)) = \lambda_s(x) \in S \cap I\]
Analogously, $B = S + I$ and if we write $b = s'+y'$, for certain $s'\in S$ and $y'\in I$, it holds
\begin{align*}
 x\ast b & = x \ast (s'+y') = -x + x(s'+y') - y' -s' = \\
 & = -x + xs' -x + xy' - y' - s' =  \ \text{($I$ is trivial)}\\
 & = -x + xs' - x + x + y' - y' - s' = x \ast s' \in I \cap S
\end{align*}
Since $I$ is abelian and $B = S + I$, it follows that $I \cap S$ is a normal subgroup of $(B, +)$. Therefore $I\cap S$ is an ideal of~$B$, as required. 

Assume that $S$ is a maximal subbrace of $B$, and let $L$ be an ideal of $B$ such that $I \cap S \subsetneq L \subseteq I$. By Proposition~\ref{prop:product_id-subbr}, $SL$ is a subbrace and $B = SL$, as $L$ is not contained in $S$. Therefore 
$I = SL \cap I = (S \cap I)L = L$. Hence $I/S\cap I$ is a chief factor of $B$.

Conversely, assume that $I/I\cap S$ is a chief factor and let $T$ be a subbrace of $G$ such that $S\subsetneq T \subseteq B$. Then  $S \subsetneq T = T \cap SI = S(T\cap I)$. Arguing as above, it follows that $I \cap T$ is an ideal of  $B$, and $S \cap I \subsetneq T \cap I \subseteq I$. Therefore $T\cap I = I$ and then $T = SI = B$. Hence $S$ is a maximal subbrace of $B$. 
\end{proof}

\begin{corolari}
\label{cor:soluble_minimal-compted-ideal}
Let $B$ be a soluble brace and $I$ be a minimal ideal of $B$. Suppose that there exists a subbrace $S$ such that $B = IS = I +S$ and  $\zeta(B)$ is not contained in $S$.  Then $S$ is an ideal and $B$ is isomorphic to the direct product $I \times S$.
\end{corolari}

\begin{proof}
We have that $I$ is abelian. Then, by Lemma~\ref{lema:max<->factppal}, $I \cap S$ is an ideal of $B$ and, thus, $I\cap S = 0$. Lemma~\ref{lema:max<->factppal} also yields $S$ is a maximal subbrace of $B$. By Theorem~\ref{teo:max-centvprime}, $S$ is  an ideal of $B$ as $\zeta(B)\nsubseteq S$. Hence $B$ is isomorphic to the direct product $I \times S$.
\end{proof}

\begin{lema}
\label{lema:abelian_chf}
Let $I/J$ be an abelian chief factor of a brace $B$. Then $I/J$ is either a Frattini chief factor or a complemented chief factor of $B$. If $B$ is finite, then every complement of $I/J$ is a maximal subbrace of~$B/J$.
\end{lema}

\begin{proof}
Without loss of generality we can assume that $J = 0$ and $I$ is a minimal ideal of $B$. Suppose that $I$ is not a Frattini chief factor of $B$. Then, there exists a maximal subbrace $S$ such that $I$ is not contained in $S$. Thus, $B = IS = I+S$. By Lemma~\ref{lema:max<->factppal}, $I \cap S$ is an ideal of $B$ and $I/I\cap S$ is a chief factor of $B$. Therefore $I \cap S = 0$ so $I$ is complemented in $B$ by the maximal subbrace~$S$. 

Assume that $B$ is finite and $T$ is another complement of $I$. Then $T$ is a proper subbrace of $B$. Let $S'$ be a maximal subbrace of $B$ with $T \subseteq S'$. Note that $I$ is not contained in $S'$, as otherwise $T$ would not be a complement of $I$. Arguing as above, it follows that  $S'$ also complements $I$. Hence $S' = T$.
\end{proof}

\begin{proof}[Proof of Theorem~\ref{teo:soluble-chief^maximal}]
Let $I/J$ be a chief factor of $B$. Since $B$ is soluble, we have that $I/J$ is abelian. By Lemma~\ref{lema:abelian_chf}, $I/J$ is either a Frattini chief factor or a complemented chief factor of $B$.

Suppose that $B$ is finite. Without loss of generality we can assume that $J = 0$ and then $I$ is a minimal ideal of $B$. Note that $I$ is an abelian trivial brace, i.e. $(I,+) = (I, \cdot)$ is an abelian group. Let $H$ be a characteristic subgroup of $I$. Then, $(H,+) = (H,\cdot)$ is a normal subgroup of both $(B,+)$ and $(B,\cdot)$, because both $(I,+)$ and $(I,\cdot)$ are normal subgroups of $(B,+)$ and $(B,\cdot)$, respectively. Since $I$ is an ideal, it holds that $\lambda_b|_I \in \Aut(I,+)$ and, therefore, $\lambda_b (H) \subseteq H$ for every $b\in B$. Consequently, $H$ is an ideal of $B$. By the minimality of $I$, either $H = 0$ or $H = I$. Therefore $I$ is a characteristically simple group. Since $I$ is abelian, it follows that $I$ is an elementary abelian $p$-group for some prime $p$. 

Let $S$ be a maximal subbrace of $B$ and let
\[ I_0 = 0 \subseteq I_1 \subseteq \ldots \subseteq I_n = B\]
be a chief series of $B$, with $I_i/I_{i-1}$  an abelian brace for every $1\leq i \leq n$. Then, there exists $1 \leq i_0 \leq n$ such that $SI_{i_0} = S+I_{i_0}= B$ and $SI_{i_0-1} = S + I_{i_0-1} = S$. Thus, $S/I_{i_0-1}$ is a maximal subbrace of $B/I_{i_0-1}$ and $I_{i_0}/I_{i_0-1}$ is a trivial brace such that $(S/I_{i_0-1})(I_{i_0}/I_{i_0-1}) = B/I_{i_0-1}$. According to Lemma~\ref{lema:max<->factppal}, it follows that $I_{i_0}/(S\cap I_{i_0})$ is a chief factor, which is  an elementary abelian $p$-group by the previous statement. Thus $I_{i_0-1} = S \cap I_{i_0}$. Hence $|B:S| = |I_{i_0}/I_{i_0}\cap S| $ is a power of $p$, as claimed. 
\end{proof}

\section{Soluble solutions of the YBE}

This section is devoted to prove Theorems~\ref{teo:soluble_brace<->soluble_sol}~and~\ref{teo:struct_brace_soluble->soluble} which characterise  solubility of solutions by means of solubility of braces. Applications to multipermutation solutions are also provided.

The following well-known result describes the relation between solutions of the YBE and braces.

\begin{teorema}[{\cite[Theorems 3.1 and 3.9]{GuarnieriVendramin17}}]
\label{teo:nodeg<->braces}
Let $B$ be a brace. Then
\[ r_B\colon B\times B \rightarrow B \times B,\quad r(a,b) = (\lambda_a(b), \lambda_a(b)^{-1}ab), \quad \forall\,  a,b\in B.\]
provides a solution on $B$. On the other hand, given $(X,r)$ a solution, there exists a brace structure on $G(X,r) = \langle x \in X\,|\, xy = uv,\, \text{if $r(x,y) = (u,v)$}\rangle$ such that its associated solution $r_G$ satisfies $r_G(\iota \times \iota) = (\iota \times \iota)r$, where $\iota\colon X \rightarrow G(X,r)$ is the canonical map.
\end{teorema}

\begin{nota}
\label{nota:lambda^G}
According to Theorem~\ref{teo:nodeg<->braces}, it follows that $(X,r)$ is a trivial solution if, and only if, the structure brace $G(X,r)$ is abelian isomorphic to $\mathbb{Z}^X$. In particular, $B$ is an abelian brace if, and only if, its associated solution $(B,r_B)$ is trivial.
\end{nota}

Recall that a homomorphism $f \colon (X,r) \rightarrow (Y,s)$ between solutions is a map $f\colon X \rightarrow Y$ satisfying $s(f\times f) = (f\times f)r$. Following~\cite{CedoOkninski21}, every homomorphism between solutions $f \colon (X,r) \rightarrow (Y,s)$ extends to a brace homomorphism between structure braces associated $\bar{f} \colon G(X,r) \rightarrow G(Y,s)$.

\begin{definicio}
Let $(X,r)$ be a solution. Suppose that $Y\subseteq X$ satisfies that $r(Y\times Y) = Y \times Y$. We denote $r|_Y$ the restriction of $r$ to $Y\times Y$ and we say that $(Y,r_Y)$ is a \emph{subsolution} of $(X,r_X)$.
\end{definicio}

\begin{proof}[Proof of Theorem~\ref{teo:soluble_brace<->soluble_sol}]

Let $B$ be a soluble brace and let
\[B = I_0 \supseteq I_1 \supseteq \ldots \supseteq I_n = 0\]
be an abelian ideal series of $B$. For every $1\leq k \leq n$, we take the canonical epimorphism $\rho_k\colon B \rightarrow B/I_k$. It turns out that the $\rho_k$ are epimorphisms between the associated solutions  $(B,r_B)$ and $(B/I_k, r_{B/I_k})$. Moreover,  $I_k = \Ker \rho_k \in B/\Ker\rho_k$ and 
\[ \rho_k(I_{k-1}, r_{B}|_{I_{k-1}}) = (I_{k-1}/I_k, r_{B/I_k}|_{I_{k-1}/I_k})\]
is a trivial subsolution, as $I_{k-1}/I_k$ is abelian. Therefore, $(B,r_B)$ is soluble at~$0$.

Now, assume that the associated solution $(B,r_B)$ of a brace $B$ is soluble at~$0$. Then, there exist a sequence of subsets
\[ X_t = \{0\} \subseteq X_{t-1} \subseteq X_1 \subseteq X_0 = B,\]
solutions $(Y_k, s_k)$, and epimorphisms $f_k\colon (B,r_B) \rightarrow (Y_k,s_k)$, such that $X_k\in B/\Ker f_k$ and $f_k(X_{k-1}, r_B|_{X_{k-1}})$ is a trivial subsolution of $(Y_k,s_k)$, for every $1\leq k \leq t$.

Fix $1\leq k \leq t$ and assume that $X_{k-1}$ is an ideal. There exists a brace epimorphism $\bar{f}_k\colon G(B,r_B) \rightarrow G(Y_k,s_k)$ which extends $f_k$, and by the universal property described in~\cite{CedoSmoktunowiczVendramin19}, there exists an epimorphism $g\colon G(B,r_B) \rightarrow B$ such that $g(b) = b$ for every $b\in B$. Thus, $B \cong G(B,r_B)/\Ker g$. Since $\bar{f}_k$ is an epimorphism, $\bar{f}_k(\Ker g)$ is an ideal of $G(Y_k,s_k)$, and we can consider the induced epimorphism of braces $\varphi_k\colon B \rightarrow G(Y_k,s_k)/\bar{f}_k(\Ker g)$ such that
\[\varphi_k(b) \coloneqq \bar{f}_k(b+\Ker g) = \bar{f}_k(b) + \bar{f}_k(\Ker g) = f_k(b) + \bar{f}_k(\Ker g)\ \text{for every $b\in B$}.\]

Observe that $\Ker g \cap B = 0$, and therefore, $\{f_k(0)\} = \bar{f}_k(\Ker g) \cap Y_k$. Since $0 \in X_k$ and $X_k \in B/\Ker f_k$, it follows that $X_k = \Ker \varphi_k$ is an ideal of~$B$. Moreover,  $f_k(X_{k-1},r_B|_{X_{k-1}})$ is a trivial subsolution of $(Y_k,s_k)$, which implies that $X_{k-1}/X_k$ is isomorphic to an abelian subbrace of $G(Y_k,s_k)/\bar{f}_k(\Ker g)$.

Hence, we have that
\[ 0 = X_t \subseteq X_{t-1} \subseteq \ldots \subseteq X_1 \subseteq X_0 = B\]
is an ideal series such that, for every $1\leq k \leq t$, $X_{k-1}/X_k$ is an abelian brace; that is, $B$ is a soluble brace.
\end{proof}

\begin{proof}[Proof of Theorem~\ref{teo:struct_brace_soluble->soluble}]
Let $(X,r)$ be a solution and assume that $G:= G(X,r)$ is  a soluble brace with an abelian ideal series
\[ G = I_0 \supseteq I_1 \supseteq \ldots \supseteq I_n = 0\]
such that $I_{n-1}\cap X \neq \emptyset$.  Then, for every $1 \leq k \leq n-1$, we take $X_{k} := I_k \cap X \neq \emptyset$ and $X_n := \{x_0\}$ for some $x_0\in X_{n-1}$. Thus, we obtain a chain of subsets
\[ X_n = \{x_0\} \subseteq X_{n-1} \subseteq \ldots \subseteq X_1 \subseteq X_0 = X\]

Fix $1 \leq k \leq  n-1$. Recall that the inclusion map $\iota \colon X \hookrightarrow G$ satisfies $r_G(\iota \times \iota) = (\iota\times \iota)r$, where $r_G$ is the associated solution with $G$. Then, the map $f_k\colon X \rightarrow G/I_k$, given by $f_k(x) = \iota(x)I_k$ provides an epimorphism between $(X,r)$ and $(G/I_k, r_{G/I_k})$. Observe that $X_k \in \Ker f_k$ and $f(X_{k-1},r|_{X_{k-1}}) = (I_{k-1}/I_k, r_{G/I_k}|_{I_{k-1}/I_k})$ is a trivial solution, as $I_{k-1}/I_k$ is an abelian brace.

Hence, we conclude that $(X,r)$ is soluble at $x$, for every $x\in I_{n-1}\cap X$.
\end{proof}

\bigskip

Theorems~\ref{teo:soluble_brace<->soluble_sol} and~\ref{teo:struct_brace_soluble->soluble} can by applied to the case of multipermutation solutions.

Let $(X,r)$ be a solution, given by $r(x,y) = (\lambda_x(y), \rho_y(x))$, for all $x,y\in X$. The so-called retraction relation $\sim$ on $X$ is defined as $x\sim y$ if $\lambda_x = \lambda_y$ and $\rho_x = \rho_y$.

If $[x]$ denotes the $\sim$-class of $x\in X$, then a natural induced solution $\Ret(X,r) = (X/\sim,\bar{r})$ called the \emph{retraction} of $(X,r)$ arises, where $\bar{r}$ is defined by
\[ \bar{r}([x], [y]) = ([\sigma_x(y)], [\tau_y(x)]), \quad \text{for all $[x], [y] \in X/ \sim$}.\]

We can iterate this process and define inductively
\begin{align*}
\Ret^1(X, r) & = \Ret(X,r),\\
\Ret^{n+1}(X,r) & = \Ret(\Ret^n(X,r)), \quad \text{for all $n\geq 1$.}
\end{align*}

A solution $(X,r)$ is said to be a \emph{multipermutation solution of level} $m$,  if $m$ is the smallest natural such that $\Ret^m(X,r)$ has cardinality 1.

Let $B$ be a brace. We define inductively the so-called \emph{socle series} as $
\Soc_1(B)  = \Soc(B)$ and for every $n \geq 1$,  $\Soc_{n+1}(B)$ is the ideal of $B$ such that $\Soc_{n+1}(B)/\Soc_n(B)  = \Soc\left(B/\Soc_n(B)\right)$. Then, following \cite{CedoSmoktunowiczVendramin19}, we say that $B$ has \emph{finite multipermutational level} $m$, if $m$ is the smallest natural such that the socle series sequence reaches $B$.

The following results characterise multipermutation solutions by means of the multipermutational property of braces.

\begin{teorema}[{\cite[Proposition 5.3]{BallesterEstebanPerezC24-JH}}]
\label{teo:bracemulti<->socle}
Let $B$ be a brace. The solution associated $(B, r_B)$ is a multipermutation solution of level $m$ if, and only if, $B$ has finite multipermutational level $m$.
\end{teorema}

\begin{teorema}[{\cite[Theorem 4.13]{CedoJespersKubatVanAntwerpenVerwimp23}}]
\label{teo:solmulti<->structure_multi}
Let $(X,r)$ be a solution. Then, $(X,r)$ is a multipermutation solution if, and only if, $G(X,r)$ is a multipermutational solution.
\end{teorema}

We present  proofs of Corollaries~\ref{cor:sol_brid_multip->sol_brid_sol} and~\ref{cor:sol_mult->sol_sol}

\begin{proof}[Proof of Corollary~\ref{cor:sol_brid_multip->sol_brid_sol}]
Since $(B,r_B)$ is of finite multipermutational level, Theorem~\ref{teo:bracemulti<->socle} yields that the socle series reaches $B$ at some $m \geq 1$:
\[ 0 \subseteq \Soc(B) \subseteq \ldots \subseteq \Soc_m(B) = B\]
Note that the socle series is an abelian series, as $\Soc_k(B)/\Soc_{k-1}(B) = \Soc(B/\Soc_{k-1}(B))$ and the socle of every brace is abelian. Therefore, $B$ is a soluble brace. Hence, Theorem~\ref{teo:soluble_brace<->soluble_sol} yields $(B,r_B)$ is a soluble solution at $0$.
\end{proof}


\begin{proof}[Proof of Corollary~\ref{cor:sol_mult->sol_sol}]
Write $G:= G(X,r)$ the structure brace of $(X,r)$. By Theorem~\ref{teo:solmulti<->structure_multi}, the associated solution $(G,r_G)$ is a multipermutation solution, and by Theorem~\ref{teo:bracemulti<->socle}, the socle series of $G$ reaches $G$ at some $m\geq 1$:
\[ 0 \subseteq \Soc(G) \subseteq \ldots \subset \Soc_m(G) = G\]
Therefore, $G$ is a soluble brace with an abelian series such that $\Soc(G)\cap X$ is not an empty set.  Theorem~\ref{teo:struct_brace_soluble->soluble} yields the final result.
\end{proof}

\section{A worked example}
\label{sec:exemple}

The aim of this section is to give an example of a brace of abelian type of order~$32$ that satisfies earlier definitions for solubility, but is not soluble according to our definition. 



Let $B$ be a brace. Denote $B_1 := B$ and $B_{i+1} := B_i \ast B_i$ for every $i \geq 1$. Bearing in mind the classical definition of group solubility, the following definition turns out to be natural. 

\begin{definicio}

A brace $B$ is said to be \emph{weakly soluble} if the series $\{B_n\}_{n\in \mathbb{N}}$ trivialises at some $n_0\in \mathbb{N}$ and $B_n/B_{n+1}$ is abelian for every $n\in \mathbb{N}$. 
\end{definicio}


Then, the following question naturally arise:

\begin{question}
Are solubility and weak solubility of braces equivalent definitions? If not, can these definitions be equivalent in braces of $\mathfrak{X}$-type for some specific class of groups $\mathfrak{X}$? Natural candidates are braces of soluble type, or more specific, braces of abelian type.
\end{question}

We answer negatively the question by constructing an example of a brace of abelian type which is weakly soluble but not soluble. This example is especially interesting because, such brace has exactly three ideals and just one maximal subbrace, and so it is very close to be simple. 



\begin{exemple} Let $B=\langle a\rangle \times \langle b\rangle \times \langle c \rangle \times \langle d\rangle\cong C_4\times C_2\times C_2\times C_2$, whose operation is written additively. The group $B$ has automorphisms $e$, $f$, $h$ given by
\begin{align*}
  e\colon a&\longmapsto a+c+d&f\colon a&\longmapsto a+b+c&h\colon a&\longmapsto a\\
  b&\longmapsto 2a+c&b&\longmapsto 2a+b& b&\longmapsto b\\
  c&\longmapsto b&c&\longmapsto c&c&\longmapsto c\\
  d&\longmapsto 2a+c+d&d&\longmapsto c+d&d&\longmapsto 2a+d.
\end{align*}
Note that $h^2=1$, $f^2=e^4$, $e^8=1$, $hfh^{-1}=f$, $heh^{-1}=e^5$, and $fef^{-1}=e^{-1}$. It follows that the group $C=\langle e,f,h\rangle$ has order~$32$ and that every element of $C$ can be written in the form $e^rf^sh^t$ with $0\le r\le 1$, $0\le s\le 1$, $0\le t\le 7$. We can construct the semidirect product $G = [B]C$ with respect to this action, whose elements will be written in multiplicative notation like in \cite{BallesterEsteban22}. We will also call $e$, $f$, $h$, $a$, $b$, $c$, $d$ the images of the corresponding elements of $C$ and $B$ by the embeddings of $C$ and $B$ in $G$. \emph{From now on, elements in $B$ are written with additive notation whenever are considered in the additive group of $B$ and with multiplicative notation whenever are considered in $G$. This difference does not suppose a notation clash.}

Let $D=\langle a^3be, af, ch\rangle \le G$. We observe that the image of $D$ under the natural projection from $G$ to $C$ coincides with $C$. Therefore, $\lvert D\rvert \ge \lvert C\rvert$. Furthermore,
\begin{align*}
  (a^3be)^2&=a^3bea^3be=a^3ba^3cda^2ce^2=bde^2,\\
  (a^3be)^4&=bde^2bde^{-2}e^4=bda^2bbce^4=a^2bce^4,\\
  (a^3be)^8&=a^2bce^4a^2bce^{-4}=1,\\
  (a^3be)^5&=a^3bea^2bce^4=a^3ba^2a^2cbe^5=a^3ce^5,\\
  (a^3be)^7&=bde^2a^3ce^5=bdaba^2ce^7=a^3cde^7,\\
  (af)^2&=afaf=aabcf^2=a^2bce^4=(a^3be)^4,\\
  (ch)^2&=chch=1,\\
  (af)(a^3be)(af)^{-1}&=a(a^3ba^3bce^7)a^{-1}=a(ace^7)a^{-1}=a^2ca^3e^7=(a^3be)^7,\\
  (ch)(a^3be)(ch)^{-1}&=c(a^3be^5)c^{-1}=a^3bce^5c=a^3bcece^{-1}e^5\\&=a^3bcbe^5=a^3ce^5=(a^3be)^5,\\
  (ch)(af)(ch)&=c(af)c^{-1}=cacf=af.
\end{align*}
We conclude that $D$ satisfies the same relations as $C$ and so $D$ is isomorphic to a quotient of $C$. Consequently, $D$ is a group of order~$32$ isomorphic to~$C$.

Our next step is to prove that the intersection of $D$ and $C$ (regarded as a subgroup of $G$) is trivial. In order to do so, it will be enough to note that the intersection contains no elements of order~$2$. It is clear that $e^4$ is the only element of order~$2$ in $\langle e\rangle$. Suppose now that we have an element of order~$2$ in $C$ of the form $e^rh$. Since $e^rhe^rh=e^{6r}h^2=e^{6r}=1$, we conclude that $4\mid r$. This gives the elements $h$, $e^4h$. Suppose now that we have an element of the form $e^rf$ of order~$2$. Then $e^rfe^rf=e^re^{-r}f^2=e^4\ne 1$. Suppose now that $e^rfh$ is an element of $C$ of order~$2$. Then $e^rfhefh=e^rfe^{5r}fh^2=e^re^{-5r}f^2=e^{-4r}e^4=e^{4(1-r)}= 1$, which implies that $8\mid 4-4r$, that is, $2\mid 1-r$. This gives the elements $efh$, $e^3fh$, $e^5fh$, $e^7fh$ of order~$2$. We must prove that the corresponding elements in $D$ under the natural isomorphism between $C$ and $D$ are different. It is clear that $(a^3be)^4=a^2bce^4\ne e^4$ and $ch\ne h$. Moreover,
\begin{align*}
  (a^3be)^4(ch)&=a^2bce^4ch=a^2bcce^4h=a^2be^4h\ne e^4h,\\
  (a^3be)(af)(ch)&=a^3becafh=a^3bbacdefh=cdefh\ne efh,\\
  (a^3be)^3(af)(ch)&=(a^3be)^2cdhfe=bde^2cdefh=bda^2cbcde^3fh\\&=a^2e^3fh\ne e^3fh,\\
  (a^3be)^5(af)(ch)&=(a^3be)^2a^2e^3fh=bde^2a^2e^3fh=a^2bdhfe^5\ne hfe^5,\\
  (a^3be)^7(af)(ch)&=bde^2a^2bde^5hf=bda^2a^2bcde^7fh=bce^7fh\ne e^7fh.
\end{align*}
We conclude that $D\cap C=\{1\}$. As mentioned before, if we write $D=\{\delta(c)c\mid c\in C\}$, then this leads to a bijective $1$-cocycle $\delta\colon C\longrightarrow B$ with allows us to construct a  brace structure on $B$.

For the sake of completeness, we include in Table~\ref{tab-deriv} all values of the bijective $1$-cocycle associated to this  brace of abelian type.
\begin{table}
  \caption{Bijective $1$-cocycle associated to the brace $(B,+,\cdot)$}
  \label{tab-deriv}
\small
  \[
  \begin{array}{llllllll}
    x&\delta(x)&x&\delta(x)&x&\delta(x)&x&\delta(x)\\\hline
    1&0&
    h&c&
    f&a&
    fh&a+c\\
    e&3a+b&
    eh&3a&
    ef&b+c+d&
    efh&c+d\\
    e^2&b+d&
    e^2h&2a+b+c+d&
    e^2f&3a+d&
    e^2fh&a+c+d\\
    e^3&3a+b+d&
    e^3h&a+d&
    e^3f&b&
    e^3fh&2a\\
    e^4&2a+b+c&
    e^4h&2a+b&
    e^4f&a+b+c&
    e^4fh&a+b\\
    e^5&3a+c&
    e^5h&3a+b+c&
    e^5f&2a+d&
    e^5fh&2a+b+d\\
    e^6&2a+c+d&
    e^6h&d&
    e^6h&3a+b+c+d&
    e^6fh&a+b+d\\
    e^7&3a+c+d&
    e^7h&a+b+c+d&
    e^7f&2a+c&
    e^7fh&b+d\\\hline
  \end{array}\]
\end{table}

We also observe that these computations can be done directly with a computer algebra system like \textsf{GAP} \cite{GAP4-13-1}. In the \textsf{YangBaxter} library of Vendramin and Konovalov \cite{VendraminKonovalov23-YangBaxter-0.10.3}, this brace is listed as \texttt{SmallBrace(32, 24003)}. 

Our next step is to determine the ideal structure of the brace. We start by classifying left ideals, which can be seen  here as $\lambda$-invariant subgroups of $B$ which are invariant under the conjugation by $C$ in $G$. 

Let us begin with all subgroups of order~$2$. We consider an element of $B$ of the form $ra+sb+tc+ud$ with $0\le r\le 3$, $0\le s\le 1$, $0\le t\le 1$, $0\le u\le 1$ fixed under the action of~$C$. Then, under the action of~$e$,
\begin{align*}
  e(ra+sb+tc+ud)&=r(a+c+d)+s(2a+c)+tb+u(2a+c+d)\\&=(r+2s+2u)a+tb+(s+u)c+(r+u)d\\
  &=ra+sb+tc+ud.
\end{align*}
It follows that $4\mid 2s+2u$, $2\mid t-s$, $2\mid s+u-t$, $2\mid r$, consequently $2\mid t$, $2\mid s$, $2\mid r$, and $2\mid u$. Therefore, $ra+sb+tc+ud\in\langle 2a\rangle$. Since $f(2a)=h(2a)=2a$, we obtain that the unique $C$-invariant subgroup of $B$ of order~$2$ is $\langle 2a\rangle$.

Now suppose that $B$ has a subgroup $L$ of order~$4$ that is invariant under the conjugation by~$C$ in~$G$. We observe that the number of elements of $L$ of order~$2$ must be odd and so, by the action of $C$, there should be at least an element of order~$2$ in $L$ that is fixed under the action of~$C$. This element must be $\langle 2a\rangle$. Now we can consider the action of $C$ on $B/\langle 2a\rangle$. The same argument as before shows that there is at least an element of order~$2$ of $B/\langle 2a\rangle$ fixed under the action of $C$. Furthermore, the action of $C$ on $B/\langle 2a\rangle=\langle \bar a, \bar b, \bar c,\bar d\rangle$ can be described as follows:
\begin{align*}
  e(\bar a)&=\bar a+\bar c+\bar d,&f(\bar a)&=\bar a + \bar b + \bar c, & h(\bar a)&=\bar a\\
  e(\bar b)&=\bar c,&f(\bar b)&=\bar b,&h(\bar b)&=\bar b\\
  e(\bar c)&=\bar b,&f(\bar c)&=\bar c,&h(\bar c)&=\bar c\\
  e(\bar d)&=\bar c+\bar d,&f(\bar d)&=\bar c + \bar d,& h(\bar d)&=\bar d.
\end{align*}
Now take an element of order~$2$, $r\bar a+s\bar b+t\bar c+u\bar d$ of $B/\langle 2a\rangle$, fixed under the action of $C$, with $r$, $s$, $t$, $u\in\{0,1\}$. Then
\begin{align*}
  e(r\bar a +s\bar b+t\bar c+u\bar d)&=r(\bar a+\bar c+\bar d)+s\bar c+t\bar b+u(\bar c+\bar d)\\
                                     &=r\bar a +t\bar b+(r+s+u)\bar c+(r+u)\bar d\\
                                     &=r\bar a+t\bar b+(r+s+u)\bar c+(r+u)\bar d,
\end{align*}
and so $2\mid s-t$, $2\mid r+s+u-t$, $2\mid r$, which implies that $2\mid r$, $2\mid u$ and $2\mid s-t$. Hence this element belongs to $\langle \bar b+\bar c\rangle$. Clearly, this element is fixed under the action of~$G$ and so the unique $C$-invariant subgroup of $B$ of order~$4$ is $\langle 2a, b+c\rangle$.

Let us compute now the $C$-invariant subgroups of order~$8$ of~$B$. Let $M$ be one of these subgroups. A similar argument as before shows that $M$ should contain $\langle 2a\rangle$ and that $M/\langle 2a\rangle$ contains the unique subgroup $C$-invariant subgroup of order~$2$ of $B/\langle 2a\rangle$. Consequently, $\langle 2a,b+c\rangle\le M$. Hence $C$ acts on $M/\langle 2a,b+c\rangle=\langle \tilde a, \tilde b, \tilde d\rangle$ by means of
\begin{align*}
  e(\tilde a)&=\tilde a+\tilde b+\tilde d,&f(\tilde a)&=\tilde a,&h(\tilde a)&=\tilde a,\\
  e(\tilde b)&=\tilde b,&f(\tilde b)&=\tilde b,&h(\tilde b)&=\tilde b,\\
  e(\tilde d)&=\tilde b+\tilde d,&f(\tilde d)&=\tilde b+\tilde d,&h( \tilde d)&=\tilde d.
\end{align*}
Now let us consider an $h$-invariant element of $B$ of the form $r\tilde a+s\tilde b+u\tilde d$, with $r$, $s$, $u\in\{0,1\}$. We have that
\begin{align*}
  e(r\tilde a+s\tilde b+u\tilde d)&=r(\tilde a+\tilde b+\tilde d)+s\tilde b+u(\tilde b+\tilde d)\\
  &=r\tilde a+(r+s+u)\tilde b+(r+u)\tilde d=r\tilde a+s\tilde b+u\tilde d.
\end{align*}
We conclude that $2\mid r+u$ and $2\mid r$, consequently this element belongs to $\langle\tilde b\rangle$. As this element is clearly $C$-invariant, we obtain that the only $C$-invariant subgroup of $B$ is $\langle 2a, b,c\rangle$.

Now we compute the $C$-invariant subgroups of order~$16$ of~$B$. An argument similar to the one used for the $C$-invariant subgroups of order~$16$ shows that they should contain $\langle 2a,b,c\rangle$. The action of $C$ on $B/\langle 2a,b,c\rangle=\langle \hat a, \hat d\rangle$ is given by
\begin{align*}
  e(\hat a)&=\hat a+\hat d,&f(\hat a)&=h(\hat a)=\hat a,\\
  e(\hat d)&=\hat d,&f(\hat d)&=h(\hat d)=\hat d.
\end{align*}
Let $r\hat a+u\hat d$ be a $C$-invariant element of order~$2$ of $B/\langle 2a,b,c\rangle$, with $r$, $u\in\{0,1\}$. Then
\begin{align*}
  e(r\hat a+u\hat d)=r\hat a+(r+u)\hat d=r\hat a+u\hat d,
\end{align*}
which implies that $2\mid r$ and so this element belongs to $\langle \hat d\rangle$. Therefore, the unique $C$-invariant subgroup of $C$ of order~$16$ is $\langle 2a, b, c, d\rangle$.

\paragraph{Left ideals of $B$.} We conclude that the left ideals of this brace, regarded in the additive group, are $L_0 := \{0\}$, $L_1 := \langle 2a\rangle$, $L_2:=\langle 2a,b+c\rangle$, $L_3:=\langle 2a, b, c\rangle$, $L_4:=\langle 2a,b,c,d\rangle$, $L_5 := B$.

\bigskip

Our next step is to compute which of these left ideals correspond are in fact ideals. We observe that
\begin{align*}
  (a^3be)^3(af)(ch)&=(bde^2)(a^3be)(ac)(fh)=bde^2a^3babcd efh=bde^2cdefh\\
                   &=bda^2cbcde^3fh=a^2e^3fh,
\end{align*}
that is, $\delta(e^3fh)=2a$. If $e^3fh$ belongs to a normal subgroup $N$ of $C$, then so does $fe^3fhf^{-1}=e^5fh$ and, consequently, $e^2\in N$. Consequently, $\langle 2a\rangle$ cannot correspond to an ideal of the brace. Moreover,
\[(a^3be)^2=a^3bea^3be=a^3ba^3cda^2ce^2=bde^2,\]
which implies that $\delta(e^2)=b+d$ and so if $2a$ belongs to an ideal, so does $b+d$. The smallest left ideal of $B$ containing $2a$ and $b+d$ must be $\langle 2a, b, c, d\rangle$. As $\delta(e^3fh)=2a$, $\delta(h)=c$, $\delta(e^2)=b+d$, and $\delta(e^4)=2a+b+c$, we have that $E=\langle e^3fh, h, e^2, e^4\rangle=\langle e^2, ef, h\rangle$ has a normal subgroup $\langle e^2\rangle$ of order~$4$ whose quotient is isomorphic to an elementary abelian group of order~$4$, that is, $\lvert E\rvert=16$. As $\lvert C\rvert=32$, $E$ is a normal subgroup of~$C$. 

\paragraph{Ideals of $B$.} We conclude that the unique non-trivial ideal of $B$ is $I_1:= \langle 2a,b,c,d\rangle = L_4$, regarded as a generating subgroup of the additive group.

\paragraph{$B$ is weakly soluble.} We compute the series of iterated star products: $B_1 = B$, $B_{n+1}:= B_n\ast B_n$, for every $n\geq 1$.

We first compute $B*B$. We note that this is an ideal of $B$ containing $\langle c+d, 2a+b+c, b+c, 2a+c, b+c, 2a, c, 2a\rangle=\langle 2a, b,c,d\rangle$, and so it is easy to see that $B*B=\langle 2a,b,c,d\rangle = I_1$.

Now let us compute $I_1*I_1$. First of all, we remember that the multiplicative group of $B*B = I_1$ corresponds to $E=\langle e^2, ef, h\rangle$. Since
\begin{align*}
  e^2(2a)&=e(2a)=2a&ef(2a)&=e(2a)=2a,\\
  e^2(b)&=e(2a+c)=2a+b,&ef(b)&=e(2a+b)=c,\\
  e^2(c)&=e(b)=2a+c,&ef(c)&=e(c)=b,\\
  e^2(d)&=e(2a+c+d)=b+c+d,&ef(d)&=e(c+d)=2a+b+c+d,
\end{align*}
and
\begin{align*}
  h(2a)&=2a,&
  h(b)&=b,\\
  h(c)&=c,&
  h(d)&=2a+d,
\end{align*}
we see that $I_1*I_1=\langle 2a, b+c\rangle = L_2$. The smallest ideal of $B$ containing $I_1*I_1$ is again $I_1$.

Finally, let us compute $L_2*L_2$.
Since $\delta(e^3fh)=2a$, $\delta(e^4)=2a+b+c$, and $\delta(e^7fh)=bc$, we see that
\begin{align*}
  e^3fh(2a)&=2a,&e^4(2a)&=2a,\\
  e^3fh(b+c)&=e^3f(b+c)=e^3(2a+b+c)=b+c,&e^4(b+c)&=b+c,
\end{align*}
we obtain that $\langle e^7fh, e^4\rangle$ acts trivially on $\langle 2a, b+c\rangle$ and so $L_2*L_2=\{0\}$. 

\paragraph{Derived ideal and subideal series.} We have obtained that $B$ has an ideal derived series
\[B\supseteq B*B= I_1 = \langle 2a,b,c,d\rangle\]
terminating at $I_1$, of order~$16$, and a subideal derived series
\begin{align*}
B\supseteq B*B&=\langle 2a, b,c,d\rangle\supseteq (B*B)*(B*B)=\langle 2a,b+c\rangle\\&\supseteq ((B*B)*(B*B))*((B*B)*(B*B))=\{0\}
\end{align*}
terminating at $\{0\}$. Hence, 
\paragraph{$B$ is a weakly soluble brace which is not soluble.}

Moreover, $I_1 = B*B$ is a chief factor which is not isomorphic to a elementary abelian $p$-group for any prime $p$. Therefore, Theorem~\ref{teo:soluble-chief^maximal} does not hold for $B$.

\bigskip

\paragraph{Maximal subbraces of $B$.} Our next goal is to check that this left brace possesses a unique maximal left subbrace, namely $I_1$. It is enough to show that if a subbrace of $B$ contains an element of additive order~$4$, then this subbrace coincides with $B$. Let $R$ be the multiplicative group of the subbrace generated
by $k$ and let $S$ be the additive group of the subbrace generated by~$k$. An element of additive order~$4$ corresponds to an element of the multiplicative group that does not belong to $\langle e^2, ef, h\rangle$. Observe that $2k=2a=\delta(e^3fh)$.

Suppose that $f\in R$, then $\delta(f)=a\in S$, and so $2a\in S$ and $3a\in S$. In particular, $f\in R$, $eh \in R$, and $e^3fh\in R$. It follows that $e^3h\in R$, and so $a+d\in S$, $e^2\in R$, and so $b+c\in S$, and $e^5fh\in R$, which implies that $2a+b+d\in S$. It follows that $a$, $b$, $c$, $d\in S$ and so $B=S$.

Suppose that $e^{2m+1}\in R$ for a given integer $m$. Then $e\in R$ and $e^3\in R$. It follows that $\delta(e)=3a+b\in S$ and $\delta(e^3)=3a+b+d\in S$, which implies that their difference $d\in S$. As $\delta(e^6h)=d$, we have that $e^6h\in R$ and so $h\in R$. Since $2(3a+b)=2a=\delta(e^3fh)\in S$, we have that $e^3fh\in R$ and we conclude that $f\in R$. In particular, $S=K$ and $R=C$.

If $e^{2m}fh\in R$ for a given integer $m$, we obtain that $e^{2m+1}fh(e^3fh)^{-1}=e^{2m-3}\in R$ and so $e\in R$, so we are in the previous case.

Suppose now that $e^{2m}f\in R$ for an integer $m$. Then $e^{2m}f(e^3fh)^{-1}=e^{2m}fh^{-1}f^{-1}e^{-3}=e^{2m}he^{-3}=e^{2m-15}h=e^{2m+1}h\in R$. Since $(e^{2m+1}h)^2=e^{6(2m+1)}h^2=e^{12m+6}=e^{4n+6}\in R$ and $(e^{2m}f)^2=e^4\in R$, we conclude that $e^6\in R$ and so $e^2\in R$. Since $e^{2m}f\in R$, we obtain that $f\in R$ and the situation is reduced to a previous case.

Finally, suppose that $e^{2m+1}h\in R$ for an integer $m$. Then $e^3fh\in R$ and $e^3fh(e^{2m+1}h)^{-1}=e^3fe^{-(2m+1)}=e^{3+2m+1}f=e^{2m+4}f\in R$ and the situation is reduced to the previous case.

We conclude that the unique maximal subbrace of $B$ is $\langle 2a,b,c,d\rangle$.

\paragraph{Minimal example.} We have checked with \textsf{GAP} \cite{GAP4-13-1} and its \textsf{YangBaxter} library \cite{VendraminKonovalov23-YangBaxter-0.10.3} that all weakly soluble skew braces of order up to~$31$ are soluble.

\end{exemple}
%

\bibliographystyle{plain}
\bibliography{bibgroup}

\end{document}